\documentclass[a4paper,11pt]{article}
\usepackage{amsmath}
\usepackage{amssymb}
\usepackage{theorem}
\usepackage{pstricks}
\usepackage{euscript}
\usepackage{epic,eepic}
\usepackage{graphicx}
\PassOptionsToPackage{normalem}{ulem}
\usepackage{graphicx}
\usepackage{setspace}
\PassOptionsToPackage{normalem}{ulem}

\newcommand{\scal}[2]{\left\langle{#1}\mid {#2} \right\rangle}

\newcommand{\HH}{\ensuremath{\mathcal H}}
\newcommand{\GG}{\ensuremath{\mathcal G}}
\newcommand{\Y}{\ensuremath{\mathcal Y}}

\newcommand{\X}{\ensuremath{\mathcal X}}

\newcommand{\KK}{\ensuremath{\mathcal K}}

\newcommand{\NN}{\ensuremath{\mathbb N}}

\newcommand{\prox}{\ensuremath{\operatorname{prox}}}

\newcommand{\rh}{\ensuremath{{\mathrm  a}}}

\newcommand{\E}{\ensuremath{\mathsf{E}}}

\newcommand{\FF}{\ensuremath{\EuScript F}}

\newcommand{\Id}{\ensuremath{\operatorname{Id}}}

\newtheorem{theorem}{Theorem}[section]

\theoremstyle{plain}{\theorembodyfont{\rmfamily}
}
\theoremstyle{plain}{\theorembodyfont{\rmfamily}
}
\theoremstyle{plain}{\theorembodyfont{\rmfamily}
\newtheorem{algorithm}[theorem]{Algorithm}}
\theoremstyle{plain}{\theorembodyfont{\rmfamily}
}
\theoremstyle{plain}{\theorembodyfont{\rmfamily}
\newtheorem{problem}[theorem]{Problem}}
\theoremstyle{plain}{\theorembodyfont{\rmfamily}
\newtheorem{remark}[theorem]{Remark}}
\theoremstyle{plain}{\theorembodyfont{\rmfamily}
}

\numberwithin{equation}{section}

\begin{document}
\title{\sffamily 
Stochastic Inertial primal-dual algorithms}
\author{Lorenzo Rosasco$^{1,2}$, Silvia Villa$^2$ and
B$\grave{\text{\u{a}}}$ng C\^ong V\~u$^2$
\\[5mm]
\small
\small $\!^1$ DIBRIS, Universit\`a degli Studi di Genova\\
\small  Via Dodecaneso 35, 16146, Genova, Italy\\
\small \ttfamily{lrosasco@mit.edu}
\\[5mm]
\small
\small $\!^2$ LCSL, Istituto Italiano di Tecnologia\\
\small and Massachusetts Institute of Technology,\\
\small Bldg. 46-5155, 77 Massachusetts Avenue, Cambridge, MA 02139, USA\\
\small \ttfamily{\{Silvia.Villa,Cong.Bang\}@iit.it}
}
\date{~}

%


\maketitle

\begin{abstract}
We propose and study a  novel stochastic inertial primal-dual approach to solve composite optimization 
problems. These latter problems arise naturally  when learning  with penalized regularization schemes.
Our analysis provide convergence results  in a general setting, that allows to analyze in a unified framework
a variety of special cases of interest. Key in our analysis is considering the framework of  splitting algorithm for 
solving a monotone inclusions in  suitable product spaces and   for a specific choice of preconditioning operators. 
\end{abstract}

\section{Introduction}
Incorporating prior information about the problem at hand is  key to learn from complex high dimensional data. 
In a variational regularization framework, a learning solution is found solving a composite optimization problem, given by an error term and a suitable {\em regularizer} \cite{SteiChri08}. 
It is the design of this latter term that allows to incorporate the prior information available. Indeed, this observation has recently lead to the study 
of vast families of regularizers \cite{BacJenMai12,ZenFig14}.

From an optimization perspective, the problem arises of devising  strategies to solve optimization problems induced 
by  general regularizers (and  error terms). While such problems might in general be non smooth, the
composite structure (the functional to be minimized is  a sum of terms composed with linear operators) 
can be exploited  considering  splitting techniques \cite{livre1,MRSVV10}. 
In particular, first order primal-dual methods have  been recently applied to a variety machine learning and  signal processing  
problems, and shown to provide state of the art results  in large scale  composite optimization problems \cite{ChaPoc11,esser2010general}.
Interestingly, the convergence of most of these methods can be analyzed within  a common framework. Indeed, 
many different algorithms can be seen as instances of a splitting approach for 
solving, so called,  monotone inclusions in  suitable product spaces and   for a specific choice of preconditioning operators. 
Taking this perspective a  unified convergence analysis  can be established  in a  Hilbert space setting. 
The price payed for this generality is that rates of convergence are not  be possible to obtain \cite{livre1}.  

In this paper,  we are interested in developing stochastic extensions of inertial primal-dual approaches for composite optimization.  
This  question is of interest when  only an uncertain/partial knowledge of the  functional to be minimized \cite{KusYin97} is available, but  
also  to consider randomized approaches to deterministic optimization problems. 
While there a few recent  studies deal with the analysis of stochastic primal dual methods in the learning setting for specific
problems \cite{ShaZha13,bianchi2014stochastic}, we are not aware of any study of the general stochastic and inertial versions 
of the primal-dual methods proposed in this paper.   Our main result is a convergence theorem for inertial stochastic forward-backward 
splitting algorithms with preconditioning. 

This point of view allows to directly get as corollaries convergence results for a wide class of optimization methods, 
some of them already known and used, and some of them new. 
In particular, in the proposed methods, stochastic estimates of the gradient of the smooth components are allowed, 
and both the proximity operators of the involved regularization terms and the involved  linear operators are activated independently and without inversions. 
From a technical point of view, our analysis has  three main features: 1) we consider convergence of the iterates (there is not an analogous of function 
values in the general setting) in a Hilbert space;  and 2) the step-size is bounded from below; this latter condition naturally leads to more stable 
implementations, since vanishing step-sizes create numerical
instabilities, however it  requires a vanishing condition on the stochastic errors; 3) we consider an inertial step, that in minimization cases lead to 
better convergence rates \cite{beck09}. 

The rest of the paper is organized as follows. In Section~\ref{sec:setting} we describe the setting, and some possible choices of regularization terms.
Moreover we show how the need of studying monotone inclusions naturally arise starting from minimization problems. In Section~\ref{sec:fb} we
introduce the stochastic inertial forward-backward algorithm with preconditioning and state its convergence properties. The derivation
of the novel primal-dual schemes, and the comparison with existing methods can be found in Section~\ref{sec:pd}. Finally, in Section~\ref{sec:exp}
we discuss the results of some numerical simulations. The proofs of our statements is deferred to the Appendix.

\section{Setting}
\label{sec:setting}
We consider the generalized learning model. 
Let  $\Xi$ be a measurable space  and assume there is a probability
measure $\rho$ on $\Xi$. Let $N\in\mathbb{N}^*$. The measure $\rho$ is fixed but known only 
through a training set $ (\xi_i)_{1\leq i\leq N} \in \Xi^N$ of 
samples i.i.d with respect to $\rho$.  Consider a  hypothesis space $\HH$, 
a bounded positive self-adjoint linear operator $V\colon\HH\to\HH$, and a 
loss function $\ell:\Xi\times \HH\to \left[0,+\infty\right[$. Suppose that $\ell$ 
has a Lipschitz continuous second partial derivative 
in the sense that there exists $\beta>0$
such that, for every $\xi\in\Xi$ and for every  $(w_1,w_2)\in\HH^2$, 
\begin{equation}
\label{cos1}
\big \|\nabla_{\!w} \ell(\xi,w_1) -\nabla_{\!w}\ell(\xi,w_2)\big\|\leq (1/\beta) \|w_1-w_2\|.
\end{equation}
Let $f\colon\HH\to\mathbb{R}$ be convex and lower semicontinuous. 
For every $j\in\{1,\ldots,s\}$, let $\mathcal{G}_j$
be a Hilbert space,  let $g_j\colon\mathcal{G}_j\to\left[0,+\infty\right]$ be a convex and lower 
semicontinuous function, and let $D_j\colon\HH\to \mathcal{G}_j$ 
be a linear and bounded operator. 
A key problem in this context is 
\begin{equation}
\label{e:prob}
\underset{w\in \HH}{\text{minimize}}\;\mathsf{E}\left[\ell(\xi,w)\right]+
f(w)+\sum_{j=1}^sg_j(D_j w),
\end{equation}
where expectation can be taken both with respect to $\rho$ or with respect to a uniform measure on the
training set. In the first case we obtain the regularized learning
 problem, and in the latter case we get the regularized empirical risk minimizaton problem, since
 for every $w\in\HH$,
 \begin{equation}
 \mathsf{E}\left[\ell(\xi,w)\right]=\frac 1 N\sum_{i=1}^N \ell(\xi_i,w).
  \end{equation}
Supervised learning problems correspond to the case where $\Xi=\X\times\Y$, the 
training set is $(\xi_i)_{1\leq i\leq N}= (x_i,y_i)_{1\leq i\leq N} \in (\mathcal{X}\times\mathcal{Y})^N$, $\HH$ is a
reproducing Hilbert space of functions, and, 
for every $((x,y),w)\in\Xi\times\HH$, $\ell(x,y,w)=L(y,w(x))$ for some loss function
$L\colon\Y\times\Y\to \left[0,+\infty\right[$.

The algorithms studied in this paper, can be used to directly solve the regularized {\em expected}  
loss  minimization problem~\eqref{e:prob} or  to  solve the regularized {\em empirical}  risk minimization 
problem.

The  term $\sum_jg_j\circ D_j$ can be seen as a 
regularizer/penalty encoding some prior information about the learning problem. Examples of convex, non-differentiable penalties 
include sparsity inducing penalties such as the $\ell_1$ norm, as well as more complex structured sparsity penalties \cite{MRSVV10,RosVilMos13}. 
\subsection{Structured sparsity} 
Consider the empirical risk corresponding to a linear regression problem on $\mathbb{R}^d$ with the square loss function, for 
a given training set $(x_i,y_i)_{1\leq i\leq N} \in (\mathbb{R}^d\times\mathbb{R})^N$
\begin{equation}
\label{e:model}
w\in\mathbb{R}^d\mapsto \frac{1}{N}\sum_{i=1}^N(\langle w,x_i\rangle-y_i)^2 +f(w)+\sum_{j=1}^sg_j( D_j w).
\end{equation}
Several well-known regularization strategies used in machine learning can be written as in \eqref{e:model},
for suitable convex and lower semicontinuous functions $f\colon\mathbb{R}^d\to\left[0,+\infty\right[$ and $g_j$, and linear 
operators $D_j$. For example, fused lasso regularization corresponds to  $f=\|\cdot\|_1$ and, for every $j\in\{1,\ldots,d-1\}$
$g_j\colon\mathbb{R}\to\mathbb{R},\; g_j=|\cdot|$, that has to be composed with $D_j\colon\mathbb{R}^d\to\mathbb{R}$, $D_jw=w_{j+1}-w_{j}$ \cite{TibSauRos05}.
In case of  group sparsity, we assume a collection $\{G_1,\ldots,G_s\}$ of subsets of $\{1,\ldots,d\}$ is given such that $\cup_{j=1}^s G_j=\{1,\ldots,d\}$.
A popular regularization term is $\ell^1/\ell^q$ regularization, for $q\in\left[1,+\infty\right]$. This can be obtained in our framework choosing
\[
f=0,\quad g_j=d_j \|\cdot\|_q^{1/q}, \quad D_j\colon\mathbb{R}^d\to\mathbb{R}^d, 
\]
with $\|\cdot\|_q$ the $\ell^q$ norm, and $D_j$ the canonical projection on the subspace $\{w\in\mathbb{R}^d\,:\, w_k=0\; \forall k\not\in G_j\}$ and
$(d_j)_{1\leq j\leq s}\in\mathbb{R}^s$ a vector of weights. Various grouped norms, such as graph lasso, or hierarchical  group lasso penalties,  
can be recovered choosing appropriately the groups $G_1,\ldots,G_s$ \cite{BacJenMai12}. 
The OSCAR penalty \cite{BonRei07}, which can be used as regularizer when it is known that the components of 
the unknown signal exhibit structured sparsity, but a group structure is not a priori known, can be included in our model.
More precisely, it is possible to set $f(w)=\lambda_1\|w\|_1+\lambda_2\sum_{i<j} \max\{|w_i|,|w_j|\}$.  This leads to the proximal
splitting methods as those proposed in \cite{ZenFig14}. Note that this approach would require the computation of the proximity operator
of $f$, which is not straightforward. An alternative approach  is to set $f=\lambda_1\|\cdot\|$, and, for every $(i,j)\in\{1,\ldots,d\}^2$ with $i<j$, define
 $D_{ij}\colon\mathbb{R}^d\to\mathbb{R}^2$, acting as $D_{ij}w=(w_i,w_j)$, and $g_{ij}\colon\mathbb{R}^2\to\left[0,+\infty\right[$, such that $g_{ij}(u)=\|u\|_{\infty}$.
With this choice, the algorithms developed 
in this work can be used to derive stochastic primal-dual proximal splitting methods, which differs from the ones treated in \cite{ZenFig14} and are novel also
in the deterministic case. In particular, they require only the computation of
the proximity operator of the conjugate of the function $g_{ij}$ which is the projection on the $\ell^1$ ball in $\mathbb{R}^2$. 
Latent group lasso formulations and, more generally, structured sparsity penalties defined as infimal convolutions \cite{MauPon12,VilRosMos14}, 
can also be treated with analogous definitions of $g_j$ and $D_j$.
We also mention that multiple kernel learning problems are also included in our framework \cite{MRSVV10, BacJenMai12}.

\subsection{From Problem \eqref{e:prob} to monotone inclusions}
Set 
\[
F=\mathsf{E}\left[\ell(\xi,\cdot)\right].
\]
The primal-dual methods proposed in this paper are based on the idea that problem~\eqref{e:prob} can be 
formulated as a saddle point problem
\begin{equation}
\label{e:pridu}
\min_{w\in\HH}\sup_{(v_1,\ldots,v_s)\in \GG_1\times\ldots\times G_s}F(w)+f(w)+\sum_{j=1}^s \left(\scal{D_j^*v_j}{w} -g_j^*(v_j)\right).
\end{equation}
If strong duality holds, then \cite[Proposition 19.18(v)]{livre1} implies that every solution 
$(\overline{w},\overline{v}_1,\ldots,\overline{v}_s)\in\HH\times\GG_1\times\ldots\times \GG_s$ of \eqref{e:pridu} 
satisfies
\begin{equation}
\label{e:moinc}
\begin{cases}
0\in \nabla F(w)+\partial f(w)+\sum_{j=1}^sD^*_j \overline{v}_j\\
0\in -D_j\overline{w}+\partial g_j^*(\overline{v}_j)\quad \forall j\in\{1,\ldots,s\}\\
\end{cases}
\end{equation}
We denote by $\mathcal{P}\times\mathcal{D}$ the set of solutions of \eqref{e:moinc}. 
In \eqref{e:pridu}, $\cdot^*$ denotes the adjoint of a linear operator and the conjugate of the function $g_j$ (see e.g. \cite{livre1} for the definition).
Let us define ${\GG}=\mathcal{G}_1\times\ldots\times\GG_s$, let $D\colon\HH\to{\GG}$, $(\forall w\in\HH)\;Dw=(D_1w,\ldots,D_sw)$
and $g\colon{\GG}\to[0,+\infty]$, $\forall v=(v_1,\ldots,v_s)\;$ $g(v)=\sum_{j=1}^s g_j(v_j)$.  We can rewrite the inclusion in \eqref{e:moinc} in a more compact form
in the space $\HH\times\GG$, as 
\begin{equation}
\label{eq:moincc}
(0,0) \in (\nabla F(\overline{w}),0)+(\partial f(\overline{w})+D^*\overline{v},  -D{\overline{w}}+\partial g^*(\overline{v})).
\end{equation}
The previous formulation leads to the study of a more general class of problems, which retain the same
key properties of the operators in \eqref{eq:moincc}. 
\begin{problem}
\label{problem}
Let $\mathcal{K}$ be a Hilbert space, let $A\colon\mathcal{K}\to 2^{\mathcal{K}}$ be a maximally monotone (multivalued) operator, and
let $B\colon\mathcal{K}\to\mathcal{K}$  be $\beta$-cocoercive for some $\beta\in\left]0,+\infty\right[$. The problem is to find 
$\overline{z}\in\mathcal{K}$ such that
\begin{equation}
\label{e:moincgen}
0\in (A+B)(\overline{z})
\end{equation}
under the assumption that the set of solutions  $\mathcal{P}$ of inclusion~\eqref{e:moincgen} is nonempty.
\end{problem}
We recall that an operator $A\colon\mathcal{K}\to 2^{\mathcal{K}}$ is maximally monotone if it is monotone, 
namely for every $v_1\in Az_1$ and $v_2\in Az_2$ in $\mathcal{K}$, $\langle v_1-v_2, z_1-z_2\rangle\geq 0$, and there is
not a monotone operator whose graph properly contains the graph of $A$.
An operator $B\colon\mathcal{K}\to\mathcal{K}$ is $\beta$-cocoercive if, for every $z_1$ and $z_2$ in $\mathcal{K}$  
\[
\scal{z_1-z_2}{Bz_1-Bz_2}\geq \beta \|Bz_1-Bz_2\|^2.
\]
The imposed structure allows to apply a forward-backward algorithm to the monotone inclusion in~\eqref{e:moincgen}.
Moreover, if in \eqref{eq:moincc} we define,
\begin{align*}
A& \colon (w,v)\in \HH\times \GG\mapsto (\partial f({w})+D^*{v},  -D{{w}}+\partial g^*({v}))\\
B&  \colon (w,v)\in \HH\times \GG\mapsto (\nabla F({w}),0)
\end{align*}
we get that $A$ is maximally monotone since it is the sum of a subdifferential operator (which is maximally monotone)
and a skew operator \cite[Example 20.30]{livre1}. Moreover, $B$ is cocoercive by the Baillon-Haddad theorem, 
since the gradient is assumed Lipschitz continuous. 
In the determistic case it has been shown that, by properly choosing a  metric on the product space
$\HH\times\GG$ different primal-dual algorithms for solving problem~\eqref{e:prob} can be derived
in this way \cite{ComPes12,ComVu12,Con13}.  Inertial versions of forward-backward algorithms for monotone inclusions
have been considered in \cite{LorPoc15} and their convergence has been proved. 

In the following sections we will show how to extend the analysis to the case when we have access only to a stochastic 
estimate of the operator $B$, obtaining as a result different stochastic inertial primal-dual schemes to solve
problem~\eqref{e:prob}. Key tools in the following sections will be $(I+A)^{-1}$, which is called resolvent of $A$ and is 
defined everywhere and single valued if $A$ is maximally monotone and the proximity operator, that is the
resolvent of the subdifferential of a convex function. 

\section{Stochastic Inertial Forward-backward splitting method for solving monotone inclusions}
\label{sec:fb}

While stochastic proximal gradient methods have been studied in several papers (see e.g. \cite{AtcForMou14,Duchi09,RosVilVu14a}), 
there are only  two recent preprints studying convergence of stochastic forward-backward algorithms
for monotone inclusions \cite{ComPes14,RosVilVu14}. In this section we take another step in filling the gap between the existing analysis
in  the deterministic setting \cite{ComVu12,LorPoc15} and the one available in the stochastic one. More precisely, 
we deal with stochastic inertial variants with preconditioning.
\begin{algorithm}
\label{algo}
In the setting of Problem \ref{problem}, let $U\colon\KK\to\KK$ be a self-adjoint and strongly positive operator. 
Let $\varepsilon\in\left]0,\min\{1,\beta\|U\|^{-1}\}\right[$,
let $(\gamma_n)_{n\in\NN}$ be a sequence in 
$\left[\varepsilon,(2-\varepsilon)\beta\|U\|^{-1}\right]$, 
and  let $(\alpha_n)_{n\in\NN}$ be a sequence in  $\left[0,1-\varepsilon\right]$.
Let $(r_{n})_{n\in\NN}$  be  a $\HH$-valued, square integrable random process,
let $w_{0}$ be  a $\HH$-valued, squared integrable random variable and set $w_{-1}=w_{0}$. 
Furthermore,
set 
\begin{equation}
\label{e:main1*}
(\forall n\in \NN)\quad
\begin{array}{l}
\left\lfloor
\begin{array}{l} 
z_{n} = w_{n} + \alpha_n(w_{n}-w_{n-1})\\
w_{n+1}= J_{\gamma_n U A}(z_{n}-  \gamma_n U r_{n}).
\end{array} 
\right.\\[2mm]
\end{array}
\end{equation}
\end{algorithm}

The first step of the algorithm is the inertial one, where a combination of the
last two iterates is taken. 
The operator $U$ is a preconditioner. While for general choices
of $U$, the resolvent operator $J_{\gamma_n UA}$ is not computable in closed
form, for suitable choices it allows to derive the above mentioned primal dual schemes.
In particular, we will see in the subsequent sections that $U$ will be built starting from the linear operators 
$(D_k)_{1\leq k\leq s}$.
When $r_n=Bz_n$, we are back to the deterministic inertial forward-backward algorithm 
which has been studied in \cite{LorPoc15} (see also \cite{MouOli03}).
Therefore, Algorithm~\ref{algo} is a preconditioned stochastic inertial forward-backward
method. To get convergence results, we need to impose restrictions on
the stochastic approximations of  $Bz_n$ and on the choice of the sequence $(\alpha_n)_{n\in\NN}$. 

\begin{theorem}
\label{t:1} 
Consider Algorithm~\ref{algo}, and set $(\forall n\in\NN)\; \FF_n = \sigma(w_0,\ldots,w_n)$.  
Suppose that  the following conditions are satisfied.
\begin{enumerate}
\item\label{cond:one1} 
$ (\forall n\in \NN)\;
\E[r_{n}| \FF_n] = Bw_{n}
$ a.s.
\item 
\label{cond:two2} 
$\sum_{n\in\NN}\E[ \|  r_{n}-Bw_{n}\|^2| \FF_n] < +\infty$ a.s.
\item\label{cond:three3}
$\sup_{n\in\NN}\| w_{n}-w_{n-1}\| <\infty $ a.s.
and $\sum_{n\in\NN}\alpha_n < +\infty $ a.s.
\end{enumerate}
Then, the following hold for some  a.s. $\mathcal{P}$-valued random variable 
$\overline{w}$.
\begin{enumerate}
\item \label{t:1i} $w_{n}\rightharpoonup \overline{w}$ a.s.
\item\label{t:1ia} $Bw_{n}\to B\overline{w}$ a.s.
\item\label{t:1ii}
 If $B$ is uniformly monotone at $\overline{w}$, then
$\|w_n-\overline{w}\|\to 0$ a.s.
\end{enumerate}
\end{theorem}
Condition 1 means that, for every iteration $n$, $r_n$ is an unbiased estimate of $Bw_n$. 
Moreover, Condition 2, requires the variance of the stochastic
approximation to decrease, and in particular to be summable. 
In principle this may seem a strong condition, but it is necessary to derive primal-dual stochastic algorithms.
Indeed, for such derivation, an analysis of  forward-backward with nonvanishing step-size is needed. 
This is a main difficulty to overcome, since even for minimization problems of a smooth function ($A=0$ and $B=\nabla f$ for
some function $f$), it is known that almost sure convergence of the iterates cannot be derived for fixed step-size and 
only  assuming that the variance is bounded, namely $\E[ \|  r_{n}-Bw_{n}\|^2| \FF_n] <\sigma^2$,
and there are explicit counterexamples (see e.g. \cite{KusYin97} and references therein). 
On the other hand, a constant stepsize could be used by using different stochastic approximations of the gradients, 
for instance those of IUG methods \cite{TseYun14}, see also \cite{LerSchBac12}, which indeed use an approximation
of the gradient having a smaller variance. 
In general we can only obtain weak convergence, as it usually happens in infinite dimensional spaces also for the 
deterministic implementations. Strong convergence can be obtained only additional monotonicity assumptions, that
for the case of minimization are related to uniform (or strong) convexity.
The sequence $\alpha_n$ is  required to be summable. Therefore, though the structure of the algorithm includes a stochastic
extension of the well-known Nesterov's accelerated method \cite{nesterov07}, the choice of $\alpha_n=(n-1)/(n+2)$ used in the minimization
setting, is not allowed by our theorem. 
Our methods are new even in the case in which $\alpha_n=0$. In this case there is not an inertial step, and we get the stochastic forward-backward
algorithm studied in \cite{ComPes14} and in \cite{RosVilVu14}. Here we make different assumptions with respect to both papers. Indeed, the 
analysis is in the same setting se in \cite{ComPes14}, but here we require a weaker condition on summability of the errors. With respect to \cite{RosVilVu14},
we removed the strong monotonicity assumptions on the operators, and a non-vanishing stepsize is allowed, but under a stronger conditions on the errors. 
The proof is based on showing that the sequence $(w_n)_{n\in\NN}$ is stochastic quasi-Fej\'er monotone \cite{Er68} with respect to the set 
of solutions $\mathcal{P}$.

\section{Special cases: minimization algorithms}
\label{sec:pd}
We show that the results obtained for the forward-backward algorithm obtained in the previous section
can be used to prove convergence of different classes of primal-dual algorithms, as well as previously known
algorithms for solving problem \eqref{e:prob}, and more generally, problem~\eqref{e:pridu}. 

\subsection{Preconditioned inertial stochastic forward-backward splitting}
In~\eqref{e:moincgen}, set $A=\partial \Big(w\mapsto\sum_{j=1}^s g_j(D_j w)+ f(w)\Big)$, and $r_k=\nabla \ell(w_k, \xi_k)$. 
Then, in this case we recover the inertial forward-backward splitting algorithm \cite{nesterov07,beck09}. 
As mentioned above, the conditions on $(\alpha_n)_{n\in\NN}$ do not allow the standard choices to
be made. Convergence in expectation of the objective function (without preconditioning) 
has been studied in the stochastic setting by several authors, see e.g. \cite{LinChePen14,Lan09,AtcForMou14}. 
We underline that a suitable preconditioning can significantly improve convergence results \cite{ChaPoc11a}.
\subsection{First class of primal-dual stochastic algorithms}
This class of algorithms can be seen as an inertial version of an extension to the stochastic setting of the primal-dual 
deterministic algorithms studied in \cite{Vu13,Con13,ComVu12} for solving problem~\eqref{e:pridu}.
\begin{algorithm}
For every $k\in \{1,\ldots, s\}$, let $W_{k}\colon\GG_k\to\GG_k$ and $V\colon\HH\to\HH$ be self-adjoint and strongly positive. 
Let $\varepsilon \in \left]0,1\right[$,
let $(\alpha_n)_{n\in\NN}$ be a sequence in $\left[0,1-\varepsilon\right]$.
Let $(\rh_{n})_{n\in\NN}$ be a $\HH$-valued, squared integrable random process, let
$w_{0}$  be a $\HH$-valued, squared integrable random vector, and set $w_{-1}=w_{0}$.  
Let
$v_{0}$  be a $\GG$-valued, squared integrable random vector and set $v_{-1}=v_{0}$.
Then, iterate, for every $n\in\NN$,

\label{algopd1}
\begin{equation}
 \begin{array}{|l}
u_{n} = w_{n} + \alpha_n(w_{n}- w_{n-1} )\\
\operatorname{For}\;k=1,\ldots, s\\
\quad
\begin{array}{|l}
 d_{k,n} = v_{k,n} + \alpha_n(v_{k,n}- v_{k,n-1} )\\
 v_{k,n+1}:=\prox^{W_k^{-1}}_{g^*_k}\big(d_{k,n}+W_{k}\big(D_k\big(u_n-2V\big(\sum_{k=1}^s D^*_kd_{k,n}+\rh_n\big)\big)\big)\\
\end{array}
\quad\\
w_{n+1}:=\prox_{f}^{V^{-1}}\big(u_n-V\big(\sum_{k=1}^s D^*_kd_{k,n}+\rh_n\big)\big).
\end{array}
\end{equation} 
\end{algorithm}

In the special case when $V=\tau\Id$ and, for every $k\in\{1,\ldots,s\}$, 
$W_k=\sigma_k\Id$, $\alpha_n=0$ for every $n\in\NN$, and the errors
are not stochastic errors,  Algorithm~\ref{algopd1} recovers the algorithm 
studied in \cite{Vu13} and similar algorithms in \cite{Con13}. It can be immediately 
seen that each proximity operator is activated individually and no inversion of the linear
operator $D$ is required. 

\begin{theorem}
\label{t:2}
In the setting of Algorithm~\ref{algopd1},
assume that
\begin{equation}
\label{e:steps}
\gamma=\Big(1-\Big(\sum_{k=1}^s\|W_k^{1/2}D_kV^{1/2}\|^{2}\Big)^{1/2}\Big)\|V\|^{-1}\beta >\frac1 2 
\end{equation}
and $\varepsilon<\min\{1,\gamma\}$.
Suppose that the following conditions are satisfied:
\begin{enumerate}
\item\label{ct:2i} 
$ (\forall n\in \NN)\;
\E[\rh_{n}| \FF_n] = \nabla F(u_n).
$
\item \label{ct:2ii}
$\sum_{n\in\NN}
\E[ \|\rh_{n}-\nabla F(u_{n}) \|^2
| \FF_n] < +\infty.$
\item\label{ct:2iii}
 $\sup_{n\in\NN}\| w_{n}-w_{n-1}\| < \infty $ a.s.
and  $\max_{1\leq k\leq s}\sup_{n\in\NN}\| v_{k,n}-v_{k,n-1}\|< \infty $ a.s.,
and $\sum_{n\in\NN}\alpha_n < +\infty $.
\end{enumerate}
Then the following hold for some random vector 
$(\overline{w}, \overline{v}_1,\ldots,\overline{v}_s)$,
$\mathcal{P}\times\mathcal{D}$-valued almost surely.
\begin{enumerate}
 \item
\label{t:2i} $w_{n}\rightharpoonup \overline{w}$ and 
$(\forall k\in\{1,\ldots,s\})$\; $v_{k,n}\rightharpoonup\overline{v}_k$ almost surely.
\item \label{t:2ii}
Suppose that the function $F$ 
is uniformly convex at 
$\overline{w}$ almost surely. Then 
$w_{n}\to \overline{w}$ almost surely.  
\end{enumerate}
\end{theorem}

The proof of Theorem~\ref{t:2}, whose sketch can be found in the appendix, starts from the observation
that Algorithm~\ref{algopd1} is an inertial stochastic forward-backward algorithm. 
Such algorithm is applied in $\HH\times\GG$, with $A$ and $B$ as in~\eqref{eq:moincc}, and 
preconditioning operator $U$, which is defined as the inverse of the lines operator
from $\HH\times\GG$ to $\HH\times\GG$, $(w,v)\mapsto(V^{-1}w-D^*v, (W_k^{-1}v_k-D_kw)_{1\leq k \leq s})$.

\begin{remark} 
Uniform convexity of $F$, which is an expectation, follows from uniform convexity of the loss function
with respect to the second variable. More precisely, let $\overline{w}\in\HH$. Suppose that there exists 
$\phi\colon:\left[0,+\infty\right[\to\left[0,+\infty\right]$ increasing and vanishing only at $0$ such that,
for every $\xi\in\Xi$ and for every $w\in\HH$,
\[ 
\ell(\xi,w)\geq \ell(\xi,\overline{w})+\scal{\nabla \ell(\xi,\overline{w})}{w-\overline{w}}+\phi(\|w-\overline{w}\|).
\]
Then $F$ is uniformly convex at $\overline{w}$ with modulus $\Phi$.
\end{remark}

\paragraph{Stochastic inertial Chambolle-Pock algorithm.}
In the special case when $s=1$, $\ell=0$, $V=\tau\Id$ and $W_1=\sigma\Id$, 
Algorithm~\ref{algopd1} is an inertial variant of Algorithm~1 in \cite{ChaPoc11}, 
which can be recovered by setting $\alpha_n=0$.
Since the second inequality in \eqref{e:steps} is always satisfied (in this case 
$\beta$ can be chosen arbitrarily small), the conditions on the stepsize reduce to 
\[
\tau\sigma\|D_1\|^2<1.
\]
Weak convergence of the iterates obtained here does not follow from the analysis
in \cite{ChaPoc11} for Algorithm 2, where the assumptions on the sequence $(\alpha_n)_{n\in\NN}$ 
are the typical ones for accelerated methods. 
A related algorithm, the so called PDHG, has been studied in \cite{zhu2008efficient, esser2010general}, 
which is a deterministic version of the above algorithm, and corresponds to the case $\alpha_n=0$ and 
$\ell=0$.
Finally, a preconditioned version of the primal-dual Algorithm 1 in \cite{ChaPoc11} has been studied in 
\cite{ChaPoc11a}, where the conditions on the preconditioning matrices correspond to the ones in \eqref{e:steps}.

\subsection{Second class}
In this section we suppose $f=0$ in \eqref{e:prob}.

\begin{algorithm}
\label{algopd2}
Let $V\colon\HH\to\HH$ be a bounded linear self-adjoint and strongly positive operator.
For every $k\in \{1,\ldots, s\}$, 
let $W_{k}\colon\GG_k\to\GG_k$ be linear, bounded, self-adjoint, and strongly positive.
Let $\varepsilon \in \left]0,1\right[$, and 
let $(\lambda_n)_{n\in\NN}$ be a sequence 
in $\left[\varepsilon,1\right]$, 
let $(\alpha_n)_{n\in\NN}$ be a sequence in $\left[0,1-\varepsilon\right]$.
let $(\rh_{n})_{n\in\NN}$ be a $\HH$-valued, squared integrable random process, and let
$w_{0}$  be a $\HH$-valued, squared integrable random vector and set $w_{-1}=w_{0}$.  
Let
$v_{0}$  be a $\GG_k$-valued, squared integrable random vector and set $v_{-1}=v_{0}$.
Then, iterate, for every $n\in\NN$,
\begin{equation}
\label{e:Algomain3b}
\begin{array}{|l}
u_{n} = w_{n} + \alpha_n(w_{n}- w_{n-1} )\\
\operatorname{For}\;k=1,\ldots, s\\
\quad
\begin{array}{|l}
d_{k,n} = v_{k,n} + \alpha_n(v_{k,n}- v_{k,n-1} )\\
\end{array}\\
s_{n} =  u_{n}-V\rh_{n} - V\sum_{k=1}^s D_{k}^*d_{k,n}\\
\quad\\
\operatorname{For}\;k=1,\ldots, s\\
\quad
\begin{array}{|l}
q_{k,n}= \prox_{g_{k}^{*}}^{W_{k}^{-1}}\big(d_{k,n}+W_{k}D_{k}s_{n}\big)\\
v_{k,n+1}=v_{k,n}+\lambda_{n}(q_{k,n}-v_{k,n})\\
\end{array}\\
w_{n+1} =  u_{n}-V\rh_{n}-\sum_{k=1}^sD_{k}^*q_{k,n}.
\end{array}
\end{equation} 
\end{algorithm}

\begin{theorem}
\label{mainal1a} 
In the setting of Algorithm~\ref{algopd2},
let $\beta$  be a strictly positive number such that \eqref{cos1} is satisfied.
Assume that $\sum_{k=1}^s\|W_k^{1/2}D_kV^{1/2}\|^2<1$,  that $\beta\|V\|^{-1}>1/2$,
and that $\varepsilon<\min\{1,\beta\}$. Set  $\FF_n = \sigma((w_0,v_0)\ldots,(w_n,v_n))$
and suppose that the following conditions are satisfied:
\begin{enumerate}
\item\label{cond:onei} 
$(\forall n\in \NN)\;
\E[\rh_{n}| \FF_n] =\nabla F(u_n).$
\item \label{cond:twoii}
$\sum_{n\in\NN} \E[ \| \rh_{n}-\nabla F(u_{n}) \|^2|\FF_n] < +\infty.$
\item\label{cond:twoiii}
$\sup_{n\in\NN}\| w_{n}-w_{n-1}\| < \infty $ a.s.,
$\max_{1\leq k\leq s}\sup_{n\in\NN}\| v_{k,n}-v_{k,n-1}\|< \infty $ a.s.,
and $\sum_{n\in\NN}\alpha_n < +\infty $.
\end{enumerate}
Then the following hold for some random vector 
$(\overline{w}, \overline{v}_1,\ldots,\overline{v}_s)$,
$\mathcal{P}\times\mathcal{D}$-valued almost surely.
\begin{enumerate}
 \item
\label{mainal1a:2i} $w_{n}\rightharpoonup \overline{w}$ and 
$(\forall k\in\{1,\ldots,s\})$\; $v_{k,n}\rightharpoonup\overline{v}_k$ almost surely.
\item \label{mainal1a:2ii}
Suppose that the function $F$ 
is uniformly convex at 
$\overline{w}$, then 
$w_{n}\to \overline{w}$ almost surely. 
\end{enumerate}
\end{theorem}

\paragraph{Generalized forward-backward for nonseparable penalties.}
Algorithm~\ref{algopd2} is a generalization under several aspects of the algorithm in 
\cite[equation (24)]{LorVer11}. Indeed, here we presented a convergence analysis for
a more general objective function, adding stochastic noise and an inertial step.
Moreover, Algorithm~\ref{algopd2} is a stochastic and inertial version of the algorithm 
in \cite[Proposition 4.3]{ComConPes14}.
A special case of Algorithm~\ref{algopd2} has been proposed in \cite{CheHuaZha13}, where $s=1$,
$V=\tau\Id$, and $W_1=\Id$.

\section{Numerical experiments}
\label{sec:exp}
Let $N$ and $p$ be strictly positive integers. 
Concerning the data generation protocol, the input points  $(x_i)_{1\leq i\leq N}$
 are uniformly drawn in the interval $\left[a,b\right]$ (to be specified later in the two cases we consider). 
 For a suitably chosen finite dictionary 
 of real valued functions $(\phi_k)_{1\leq k\leq p}$ 
 defined on $\left[a,b\right]$, the labels are computed using a noise-corrupted regression
 function, namely
\begin{equation}
(\forall i \in \{1,\ldots, N\})\quad y_i =  \sum_{k=1}^{p}\overline{w}_k\phi_k(x_i)+ \epsilon_i,
\end{equation}
where $(\overline{w_k})_{1\leq k\leq p} \in\mathbb{R}^p$ and $\epsilon_i$ is an additive noise $\epsilon_i\sim \mathcal{N}(0,0.3)$.

We will consider a polynomial dictionaryy of functions, i.e. 
$(\forall k\in \{1,\ldots,p\})$ $\phi_k\colon \left[-1,1\right]\to \mathbb{R}$, $\phi_k(x)=x^{k-1}$.
We estimate $\overline{w}$ by
 solving the following regularized minimization problem 
\begin{equation}
\label{e:ff1}
 \underset{ (w_k)_{1\leq k\leq p}\in\mathbb{R}^p}{\text{minimize}} \;\frac{1}{N}
\sum_{i=1}^N\Big(y_i -\sum_{k=1}^{p}w_k\phi_k(x_i)\Big)^2
+ \lambda\sum_{l=1}^s \Big(\sum_{j\in G_l}(w_j)^2\Big)^{1/2}
\end{equation}
where $\lambda$ is a strictly positive parameter.   
Problem  \ref{e:ff1} is a special case of Problem \ref{e:prob}, and hence it can be 
solved by using the stochastic inertial forward-backward  splitting (first class). 
We set 
 \begin{alignat}{2} 
&p=32, \quad s=8,\quad N=48,\quad \gamma_n  = 15/(n+100),\quad \alpha_n = \gamma^2_n,\quad 
 \lambda = 0.02, \notag\\
&\nonumber \overline{w}=[3,2,1,0,1,0,1,2,\!-1,0,0,-2,-1,1,0.5,0,1,0,4,0,\!-2,0,0,\!-2,1.0,1,0,0.2,\!-0.1,0,0,1]\\
&(\forall l\in\{1,\ldots,8\})\; G_l =[4l-3,\ldots,4l+1]
\end{alignat}
Here, we use the variants of the exact gradient for the stochastic gradient as follows 
\begin{equation}
\rh_n = \nabla F(u_n) + \mathcal{N}(0,\mathrm{Sig})/n.
\end{equation}
The resulting regression functions using the stochastic inertial primal-dual splitting (SIPDS) are shown in  Figure \ref{fig: orbit1pc1z1} (right).
To check convergence towards a solution of~\eqref{e:ff1}, we computed a solution of \eqref{e:ff1} by running the  
corresponding deterministic primal-dual  splitting method in \cite{Vu13}  for 5000 iterations. 
%
\vspace{-3cm}
\begin{figure}[!ht]
   \centering     
   \begin{tabular}{cc}
 \includegraphics[width=5.5cm]{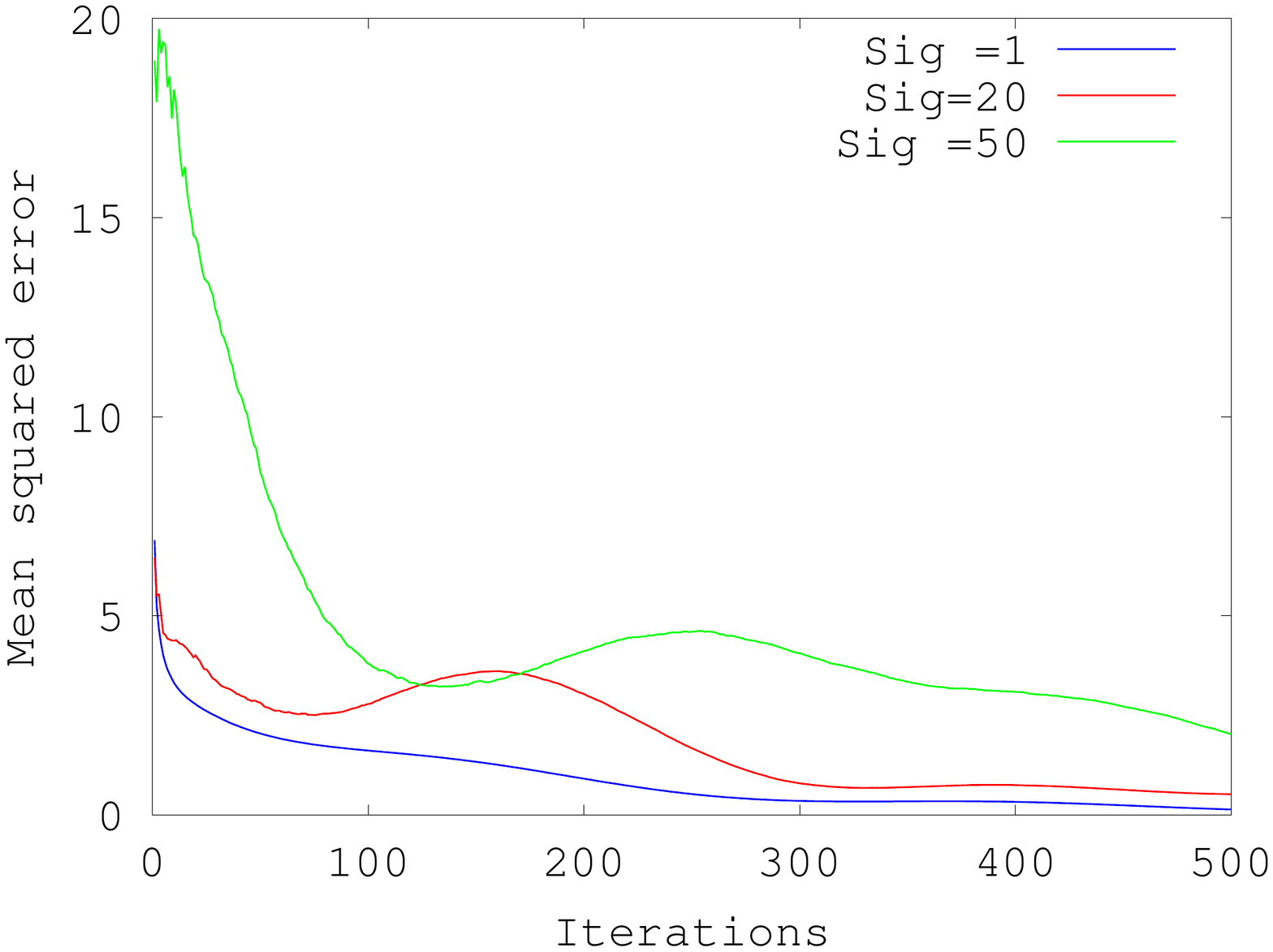}& \includegraphics[width=5.5cm]{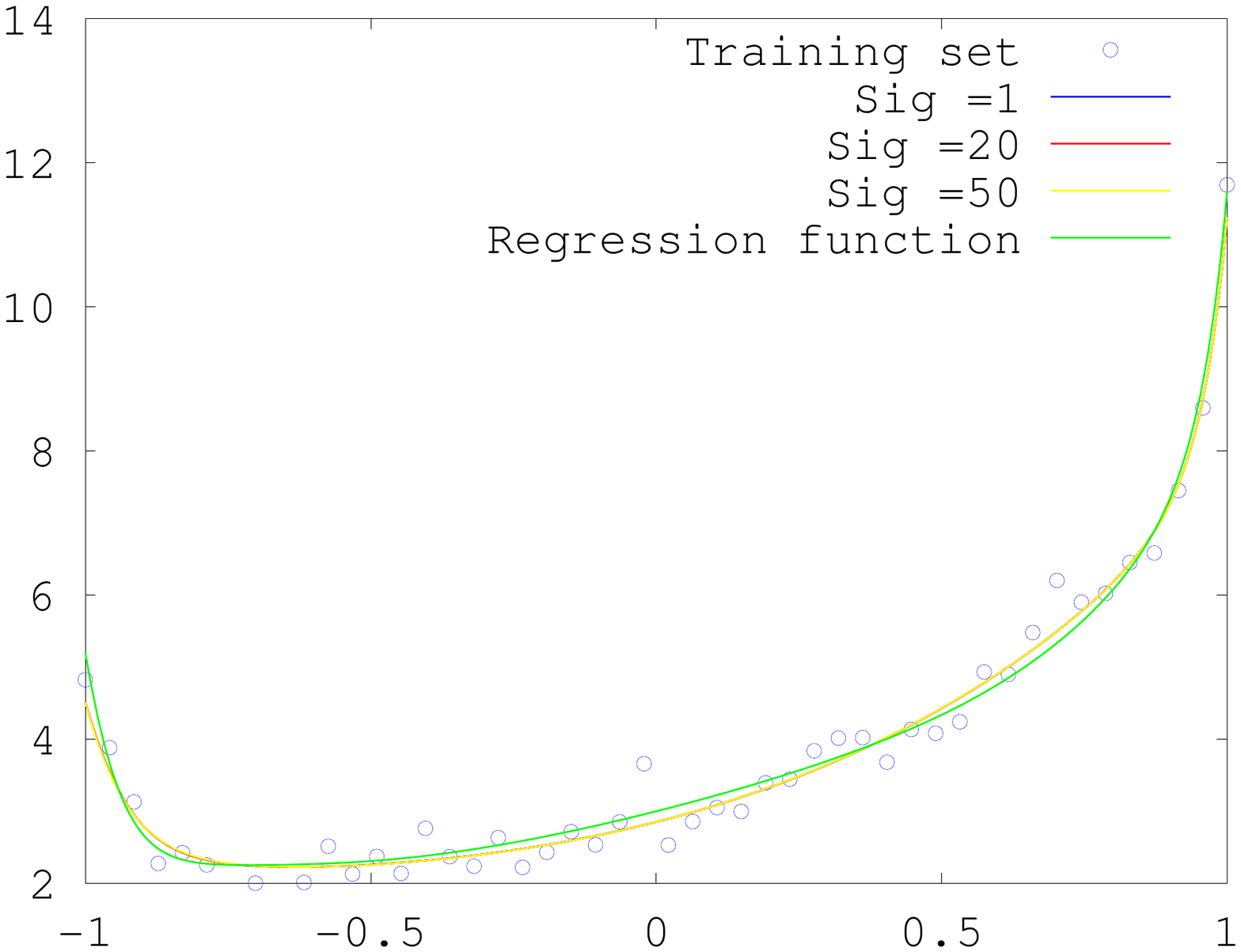}
 \end{tabular}
 \caption{Convergence of the iterates of SIPDS applied to Problem~\ref{e:ff1} (left), and corresponding approximations of regression functions
  (right). 
 \label{fig: orbit1pc1z1}}
 \end{figure}

%

{\small \bibliographystyle{plain}
\bibliography{biblio_colt}}
\appendix
\section{Proofs}
\begin{proof}[Proof of Theorem~\ref{t:1}]
Since $U$ is  self-adjoint and strongly positive, $UA$ is also maximally monotone 
by \cite[Lemma 3.7]{ComVu12}.
Since $B$ is cocoercive and has full domain, therefore it is also maximally monotone \cite[Corollary  20.25]{livre1}. 
Let $w\in\mathcal{P}$ and set
\begin{equation}
\label{e:set1}(\forall n\in \NN)\quad
u_n = z_n-w_{n+1} - \gamma_n U(r_n- Bw).
\end{equation}
Then, we have 
\begin{equation}
(\forall n\in \NN)\quad
w= J_{\gamma_n U A}(w - \gamma_nU Bw).
\end{equation}
We derive from \cite[Lemma 3.7]{ComVu12} that $J_{\gamma_nUA}$ is firmly nonexpansive
with respective to the norm $\|\cdot \|_{V}$, therefore
\begin{alignat}{2}
\label{e:est1}(\forall n\in \NN)\quad
&\|w_{n+1}-w \|_{V}^2 \leq \|z_n-w -\gamma_n U(r_n- Bw) \|_{V}^2
-  \|u_n \|_{V}^2\notag\\
& = \|z_n- w\|_{V}^2 - 2\gamma_{n}\langle{z_n-w},{r_n- B w}\rangle+ \gamma_{n}^2 \|U(r_n- B w) \|_{V}^2 -  \|u_n \|_{V}^2.
\end{alignat}
By 1, since $z_n$ is $\FF_n$-measurable, we have 
\begin{alignat}{2}
\label{e:est2}(\forall n\in \NN)\quad
\E[\langle{z_n- w},{r_n- Bw}\rangle|\FF_n] 
&= \langle{z_n-w},{Bz_n- B w}\rangle.
\end{alignat}
By the same reason, for every $n\in\NN$, since $Bz_n$ is $\FF_n$-measurable, we also have 
\begin{alignat}{2}
\label{e:est3}
\E[\|U(r_n- B w) \||_{V}^2|\FF_n] 
&= \E[\|U(r_n- B z_n) \|_{V}^2|\FF_n] +\|U(Bz_n-Bw) \||_{V}^2\notag\\
&\hspace{1.4cm}+ 2\E[\langle{Bz_n- Bw},{r_n- Bz_n}\rangle|\FF_n ] \notag\\
&=\E[\|U(r_n-Bz_n) \|_{V}^2|\FF_n] +\|U(B z_n- B w) \|_{V}^2\notag\\
&\leq \E[\|U(r_n-B z_n) \|_{V}^2|\FF_n]  +
\|U\|\beta^{-1}\langle{z_n-w},{Bz_n- B w}\rangle,
\end{alignat}
where the last inequality follows from cocoercivity of $B$.
Therefore, for every $n\in\NN$, we derive from~\eqref{e:est1}, ~\eqref{e:est2}  and ~\eqref{e:est3} that 
\begin{alignat}{2}
\E[\|w_{n+1}- w \|_{V}^2|\FF_n]
&\leq \|z_n- w\|_{V}^2 -\varepsilon\gamma_n\langle{z_n- w},{Bz_n- B w}\rangle\notag\\
&\hspace{0.5cm}+\gamma_{n}^2\E[\|U(r_n- B z_n) \|_{V}^2|\FF_n] -\E[\|u_n \|_{V}^2|\FF_n]\notag\\
&\leq \|w_n-w\|_{V}^2 +\alpha_n( \|w_n-w\||_{V}^2 - \|w_{n-1}- w\|_{V}^2)
+\zeta_n-\xi_n\\
&\leq (1 +\alpha_n)( \|w_n-w\||_{V}^2 
+\zeta_n-(\alpha_n \|w_{n-1}- w\|_{V}^2+\xi_n),
\label{e:cc1}
\end{alignat}
with 
\begin{equation}\label{eq:dddd}
(\forall n\in\NN)\quad
\begin{cases}
\zeta_n = 2\alpha_n \|w_{n-1}-w_n\||_{V}^2 + \gamma_{n}^2\E[\|U(r_n- B z_n) \|_{V}^2|\FF_n] \notag\\
\xi_n=
 \E[\|u_n \|_{V}^2|\FF_n]+ \varepsilon\gamma_n\langle{z_n- w},{Bz_n-B w}\rangle.
\end{cases}
\end{equation}
Note that, for each $n\in\NN$,  $\zeta_n$ and $\xi_n$ are non-negative and 
$\FF_n$-measurable. Moreover, $(\zeta_n)_{n\in\NN}$ is summable, and
hence, we derive from \cite[Theorem~1]{ComPes14} that 
\begin{equation}
\label{eq:aa}
\exists\; \tau =  \lim_{n\to\infty} \|w_{n}- w\|_{V}^2 \quad 
\text{and}
\quad \sum_{n\in\NN}(\alpha_n \|w_{n-1}- w\|_{V}^2 +\xi_n) < +\infty.
\end{equation}

Moreover, since $\inf \gamma_n>0$, we have
\begin{equation}\label{eq:co1}
\sum_{n\in\NN}\langle{z_n- w},{Bz_n-B w}\rangle < +\infty
\Longrightarrow \langle{w_n- w},{Bz_n- Bw}\rangle \to 0.
\end{equation} 
and 
\begin{equation}\label{eq:co2}
\sum_{n\in\NN} \E[\|u_n \|^2|\FF_n]< +\infty 
\Longrightarrow\E[\| z_n-w_{n+1} - \gamma_nU(r_n- B w)\|^2|\FF_n] \to 0.
\end{equation} 
Next, from the cocoercivity of $B$, we derive from ~\eqref{eq:co1} that 
\begin{equation}
\label{eq:co3}
Bz_n \to B w.
\end{equation} 
and we also derive from ~\eqref{eq:co2} and  \eqref{eq:co3}, and condition 2 in the statement, that
\begin{alignat}{2}\label{eq:co4}
\E[\| z_n-w_{n+1}\|^2|\FF_n] &\leq 2 \E[\| z_n-w_{n+1} - \gamma_nU(r_n- B w)\|^2|\FF_n] 
+ 2\E[\|\gamma_nU(r_n-B w)\|^2|\FF_n] \notag\\
&\leq2\bigg(\E[\| z_n-w_{n+1} - \gamma_nU(r_n- B w)\|^2|\FF_n] \bigg) + 
2\E[\|\gamma_nU(r_n- B z_n)\|^2|\FF_n] \notag\\
&\quad+2 \|\gamma_nU(Bz_n-Bw)\|^2\to 0.
\end{alignat} 
Hence, by condition 3, we obtain 
\begin{equation}\label{eq:co6}
 \E[\|r_n-Bw \|^2|\FF_n] \to 0.
\end{equation}
Now define 
\begin{equation}
(\forall n\in\NN)\quad\overline{w}_{n+1} = J_{\gamma_nA}(z_n-\gamma_nUB z_n).
\end{equation}
Then $\overline{w}_{n+1}$ is $\FF_n$-measurable 
since $ J_{\gamma_nA}\circ(\Id-\gamma_nUB)$ is continuous.
Therefore,
\begin{alignat}{2}
\label{e:ff11}
(\forall n\in\NN)\quad\|z_n-\overline{w}_{n+1} \|_{V}^2
&= \E[\|z_n-\overline{w}_{n+1} \|_{V}^2|\FF_n]\notag\\
&\leq 2\E[\|w_{n+1}-z_n \|_{V}^2|\FF_n]+2\E[\|\gamma_nU(r_n-Bz_n) \|_{V}^2|\FF_n] \to 0.
\end{alignat}

\ref{t:1i}:
Now, let $\overline{w}$ be a weak cluster point of $(w_n)_{n\in\NN}$, i.e., there exists a subsequence 
$(w_{k_n})_{n\in\NN}$ which converges weakly to $\overline{w}$. It follows from our assumption that
$ (z_{k_n})_{n\in\NN}$ converges weakly to $\overline{w}$.
By \eqref{e:ff11}, $(\overline{w}_{k_n+1})_{n\in\NN}$ converges weakly to $\overline{w}$. On the other hand, since 
$B$  is maximally monotone  and its graph is therefore 
sequentially closed in $\KK^{\text{weak}}\times\KK^{\text{strong}}$
\cite[Proposition~20.33(ii)]{livre1}, by \eqref{eq:co3},
$Bw = B\overline{w}$. By definition of resolvent operator, we have
\begin{equation}\label{eq:co7}
\frac{U^{-1}(z_{k_n}-\overline{w}_{k_{n}+1})}{\gamma_{k_n}}-Bz_{k_n} \in Aw_{{k_n}+1},
\end{equation}
and hence using the sequential closedness of the graph of $A$ in
$\KK^{\text{weak}}\times\KK^{\text{strong}}$
\cite[Proposition~20.33(ii)]{livre1},
we get $-B\overline{w} \in A\overline{w}$ or equivalently, 
$\overline{w}\in(A+B)^{-1}(\{0\})$. Therefore, every weak cluster point of $(w_n)_{n\in\NN}$
is in $(A+B)^{-1}(\{0\})$ which is non-empty closed convex \cite[Proposition~23.39]{livre1}. 
By \cite[Theorem~1]{ComPes14}, $(w_{n})_{n\in\NN}$ converges weakly to a random vector 
$\overline{w}$, taking values in $(A+B)^{-1}(\{0\})$ almost surely.\\
\ref{t:1ia}: From the cocoercivity of $B$, for every $n$ in $\NN$
\begin{alignat}{2}
 \|Bw_n-Bz_n\| &\leq \beta^{-1}\|w_n-z_n\|= \beta^{-1}\alpha_n\|w_n-w_{n-1}\| \to 0 
\end{alignat}
by \eqref{eq:aa}. By \eqref{eq:co3}, we obtain $Bw_n\to B\overline{w}$.

\ref{t:1ii}: This conclusion follows from since strong monotonicity implies demiregularity \cite[Definition 2.3]{AttBriCom10} 
and \ref{t:1ia}.
\end{proof}
Next we give a sketch of the proof for Theorem~\ref{t:2}.
\begin{proof}[Proof of Theorem~\ref{t:2}]
Let $\KK=\HH\times\GG$, and define  $A$ and $B$ as in~\eqref{eq:moincc}. Define $W\colon\GG\to\GG$ by setting $W(v_1,\ldots,v_s)=(W_1v_1,\ldots,W_sv_s)$.
Let $U^\prime\colon\KK\to\KK$ be the linear operator defined by setting $(w,v)\mapsto(V^{-1}w-D^*v, W^{-1}v-Dw)$.
Since $\|\sqrt{W}D\sqrt{V}\|<1$ by assumption, proceeding as in \cite[Lemma 4.3(i) and Lemma 4.9(i)]{PesRep14}, 
we get that $U^\prime$ is strongly positive and self-adjoint. Therefore, its inverse, denoted by $U$ is also strongly positive and self-adjoint. 
Since $B:(w,v)\mapsto(\nabla F(w),0)$, and $\nabla F$ is $\beta$ cocoercive, it follows that $B$ is $\beta\|V\|^{-1}$ cocoercive in the 
norm induced by $V$. By \cite[Lemma 4.3(ii)]{PesRep14} we also derive that $B$ is cocoercive in the norm induced by $U$ with cocoercivity constant
$\gamma=(1-\|\sqrt{W}D\sqrt{V}\|)\beta\|V\|^{-1}$.
The statement follows by noting that Algorithm~\ref{algopd1} can be equivalently written as
\begin{equation}
\label{e:mainod}
(\forall n\in \NN)\quad
\begin{array}{l}
\left\lfloor
\begin{array}{l} 
(u_{n},d_n) = (w_{n},v_n) + \alpha_n((w_{n},v_n)-(w_{n-1},v_{n-1}))\\
(w_{n+1},V_{n+1})= J_{U A}((u_{n},d_n)- U (r_{n},0))
\end{array} 
\right.\\[2mm]
\end{array}
\end{equation}
and all the assumptions of Theorem~\ref{t:1} are satisfied. 
\end{proof}
Finally, we also present the key steps to prove Theorem~\ref{mainal1a}.
The proof follows the same lines as that of Theorem~\ref{t:2}.

\begin{proof}{Proof of Theorem~\ref{mainal1a}}
Let $\KK=\HH\times\GG$, and define  $A$ and $B$ as in~\eqref{eq:moincc}. Define $W\colon\GG\to\GG$ by setting $W(v_1,\ldots,v_s)=(W_1v_1,\ldots,W_sv_s)$.
Let $T\colon\KK\to\KK\colon(w,v)\mapsto (Vw,(W^{-1}-D^*VD)^{-1}v)$. Then $T$ is strongly positive and self adjoint. Algebraic manipulations then show that with this choice
we can express Algorithm~\ref{algopd2} as 
\begin{equation}
\label{e:mainpd}
(\forall n\in \NN)\quad
\begin{array}{l}
\left\lfloor
\begin{array}{l} 
(u_{n},d_n) = (w_{n},v_n) + \alpha_n((w_{n},v_n)-(w_{n-1},v_{n-1}))\\
(w_{n+1},V_{n+1})= J_{T A}((u_{n},d_n)-T (r_{n},0)),
\end{array} 
\right.\\[2mm]
\end{array}
\end{equation}
which is a special instance of iteration \eqref{e:main1*}, with $(\forall n\in\NN)\;\gamma_n=1\in\left]\epsilon,(2-\epsilon)\beta\|T\|^{-1}\right[$
\end{proof}

\end{document}